\documentclass[a4paper,reqno]{amsart}
\usepackage{amssymb, amsmath, amscd}
\usepackage[final]{graphicx}
\usepackage{wrapfig}
\usepackage{color}
\usepackage{soul}
\date{03 November 2016}

\newtheorem{theorem}{Theorem}[section]
\newtheorem{lemma}[theorem]{Lemma}
\newtheorem{remark}[theorem]{Remark}

\newtheorem{ansatz}[theorem]{Ansatz}
\makeatletter
 \@addtoreset{equation}{section}
\makeatother
\begin{document}
\title{Geodesic completeness for Type~$\mathcal{A}$ surfaces}
\author{D. D'Ascanio, P. Gilkey, and P. Pisani}
\address{D'Ascanio: Instituto de F\'isica La Plata, CONICET,
Universidad Nacional de La Plata, CC 67 (1900) La Plata, Argentina.}
\email{dascanio@fisica.unlp.edu.ar}
\address{Gilkey: Mathematics Department, University of Oregon, Eugene OR 97403 USA.}
\email{gilkey@uoregon.edu}
\address{Pisani: Instituto de F\'isica La Plata, CONICET and Departamento de F\'isica, Fa\-cul\-tad de Ciencias Exactas,
	Universidad Nacional de La Plata, CC 67 (1900) La Plata, Argentina.}
\email{pisani@fisica.unlp.edu.ar}
\keywords{Ricci tensor, homogeneous affine surface, geodesic completeness}
\subjclass[2010]{53C21}
\begin{abstract} {Type~$\mathcal{A}$ surfaces are the locally homogeneous affine surfaces which can be locally described by constant Christoffel symbols. We address the issue of the geodesic completeness of these surfaces: we show that some models for Type~$\mathcal{A}$ surfaces are geodesically complete, that some others admit an incomplete geodesic but model geodesically complete surfaces, and that there are also others which do not model any complete surface. Our main result provides a way of determining whether a given set of constant Christoffel symbols can model a complete surface.}
\end{abstract}

\maketitle

\section{Introduction}
An {\it affine surface} is a pair $\mathcal{M}:=(M,\nabla)$ where $\nabla$ is a torsion free connection on the
tangent bundle of a smooth $2$-dimensional surface $M$. Let $\vec x=(x^1,x^2)$ be local coordinates on $M$.
Adopt the {\it Einstein convention} and sum over repeated indices to expand
$\nabla_{\partial_{x^i}}\partial_{x^j}=\Gamma_{ij}{}^k\partial_{x^k}$
in terms of the {\it Christoffel symbols} \,$\Gamma=(\Gamma_{ij}{}^k)$; the condition that
$\nabla$ is torsion free is then equivalent to the symmetry $\Gamma_{ij}{}^k=\Gamma_{ji}{}^k$.
We say that $\mathcal{M}$ is {\it locally homogeneous} if given any two points of $M$, there is the germ of a diffeomorphism
$\Phi$ taking one point to another with $\Phi^*\nabla=\nabla$. The locally homogeneous affine surfaces have
been classified by B. Opozda \cite{Op04}; we refer to T. Arias-Marco and O. Kowalski \cite{AMK08} for the corresponding classification
if the torsion tensor is permitted to be non-zero; see also
 \cite{Du, G-SG, KVOp2, KVOp, Opozda} for related work.
\begin{theorem} [Opozda]\label{T1.1}
Let $\mathcal{M}=(M,\nabla)$ be a locally homogeneous affine surface which is not flat. Then at least one of the following
three possibilities holds, which are not exclusive, and which describe the local geometry:
\smallbreak\noindent{\bf Type~$\mathcal{A}$}: There exist local coordinates $(x^1,x^2)$ so that
$\Gamma_{ij}{}^k=C_{ij}{}^k$ is constant.
\smallbreak\noindent{\bf Type~$\mathcal{B}$:} There exist local coordinates $(x^1,x^2)$ so that
$\Gamma_{ij}{}^k=(x^1)^{-1}C_{ij}{}^k$ where $C_{ij}{}^k$ is constant.
\smallbreak\noindent{\bf Type~$\mathcal{C}$}: $\nabla$ is the Levi-Civita connection of a metric of constant sectional curvature.
\end{theorem}

We say that $\mathcal{M}$ is Type~$\mathcal{A}$ or Type~$\mathcal{B}$
or Type~$\mathcal{C}$ depending on which possibility holds in Theorem~\ref{T1.1}.
There are no surfaces which are both Type~$\mathcal{A}$ and Type~$\mathcal{C}$.
There are surfaces which are both Type~$\mathcal{A}$ and Type~$\mathcal{B}$.
Any surface which is both Type~$\mathcal{B}$ and Type~$\mathcal{C}$ is modeled either on the hyperbolic plane
or on the Lorentzian hyperbolic plane.

Let $\rho(\xi,\eta):=\operatorname{Tr}(\sigma\rightarrow R(\sigma,\xi)\eta)$
be the Ricci tensor. Although in general the Ricci tensor can be non-symmetric for an affine surface, $\rho$ is symmetric if $\mathcal{M}$ is
Type~$\mathcal{A}$. The Ricci tensor encodes all the geometry in dimension $2$; $\rho=0$ if and only if the geometry is flat.
We shall restrict for the remainder of this paper to Type~$\mathcal{A}$ surfaces
and shall consider in a subsequent paper Type~$\mathcal{B}$ surfaces; since we shall always assume $\mathcal{M}$ is not flat,
the Ricci tensor is non-zero. We say that $\mathcal{M}$ is a {\it symmetric space} if $\nabla R=0$ or, equivalently since we are in
the 2-dimensional setting, if $\nabla\rho=0$.
If  $\mathcal{C}:=\{C_{ij}{}^k\}$ is a collection of constants with $C_{ij}{}^k=C_{ji}{}^k$, let
$$
\mathcal{M}_{\mathcal{C}}:=(\mathbb{R}^2,\nabla^{\mathcal{C}})\text{ where }
\nabla^{\mathcal{C}}_{\partial_{x^i}}\partial_{x^j}=C_{ij}{}^k\partial_{x^k}\,.
$$
The translations act transitively on $\mathbb{R}^2$ preserving $\nabla^{\mathcal{C}}$ so this is a homogeneous geometry. Such an affine
surface will be said to be a {\it Type~$\mathcal{A}$ model}.
The general linear group $\operatorname{GL}(2,\mathbb{R})$ acts
on the set of Type~$\mathcal{A}$ models by change of basis or, equivalently, by the linear action on the parameter space
$S^2(\mathbb{R}^2)\otimes \mathbb{R}^2$; two indices are down and one index is up. Let $\rho_{\mathcal{C}}$ be the
associated Ricci tensor. We say that
two Type~$\mathcal{A}$ models $\mathcal{M}_{\mathcal{C}}$ and $\mathcal{M}_{\tilde{\mathcal{C}}}$ are {\it linearly isomorphic} if they are
isomorphic modulo the action of $\operatorname{GL}(2,\mathbb{R})$.  Let $\mathcal{M}=(M,\nabla)$
be a Type~$\mathcal{A}$ affine surface. We can find a Type~$\mathcal{A}$ atlas of open sets $\{U_\alpha\}$, i.e.\ local coordinates
$(x_\alpha^1,x_\alpha^2)$, and constant local Christoffel symbols $\mathcal{C}_\alpha$ so that
$\nabla_{\partial x_\alpha^i}\partial_{x_\alpha^j}=C_{\alpha,ij}{}^k\partial_{x_\alpha^k}$. Since $\mathcal{M}$ is locally homogeneous,
we can take $\mathcal{C}_\alpha=\mathcal{C}$ to be independent of $\alpha$. In this setting, $\mathcal{M}_{\mathcal{C}}$ is said to be a
{\it model} for $\mathcal{M}$. We refer to \cite{BGGP16} for the proof of the following observation:

\begin{lemma}\label{L1.2}
\ \begin{enumerate}
\item If $\mathcal{M}$ is a Type~$\mathcal{A}$ affine surface, then any Type~$\mathcal{A}$ atlas is real analytic.
\item If $\mathcal{M}_{\mathcal{C}}$ and $\mathcal{M}_{\tilde{\mathcal{C}}}$ are Type~$\mathcal{A}$ models with
$\operatorname{Rank}(\rho_{\mathcal{C}})=\operatorname{Rank}(\rho_{\tilde{\mathcal{C}}})=2$, then
 $\mathcal{M}_{\mathcal{C}}$ and $\mathcal{M}_{\tilde{\mathcal{C}}}$ are isomorphic
if and only if they are linearly isomorphic.
\end{enumerate}\end{lemma}

Introduce the following Type~$\mathcal{A}$ models by giving the corresponding Christoffel symbols; in the interests of brevity
we only list the non-zero Christoffel symbols $C_{ij}{}^k$ for $i\le j$ and set $C_{ji}{}^k:=C_{ij}{}^k$ if $i>j$.
$$\begin{array}{ll}
\mathcal{C}_1:=\{C_{11}{}^1=-1,\ C_{12}{}^1=-\frac12\},&
\mathcal{C}_2:=\{C_{12}{}^1=-\frac12\},\\
\mathcal{C}_3:=\{C_{11}{}^1=-1\ C_{22}{}^1=-1\},&
\mathcal{C}_{\pm1,\delta}:=\{C_{11}{}^2=\pm1,\ C_{12}{}^1=\frac12, C_{12}{}^2=\frac\delta2\}.
\end{array}$$
Denote the corresponding Type~$\mathcal{A}$ models by $\mathcal{M}_1$, $\mathcal{M}_2$, $\mathcal{M}_3$, and $\mathcal{M}_{\pm,\delta}$,
respectively.  By replacing $x^1$ by $-x^1$, we may replace $\mathcal{C}_{\pm1,\delta}$ by $\mathcal{C}_{\pm1,-\delta}$; thus we shall assume
$\delta\ge0$. A direct computation shows $\rho_{\mathcal{C}_1}$ and $\rho_{\mathcal{C}_2}$ are negative semi-definite,
$\rho_{\mathcal{C}_3}$ is positive semi-definite, $\rho_{\mathcal{C}_{+,\delta}}$ is
negative definite, and $\rho_{\mathcal{C}_{-,\delta}}$ is indefinite.
We refer to \cite{BGGP16,BGGP16a} for the proof of the following result; it shows, in particular, that Lemma~\ref{L1.2}~(2) can fail
if the Ricci tensor has rank 1 since $\mathcal{M}_1$ is not linearly isomorphic to $\mathcal{M}_2$ but is isomorphic to $\mathcal{M}_2$.

\begin{lemma}\label{L1.3}
$\mathcal{M}_1$, $\mathcal{M}_2$, and $\mathcal{M}_3$ are symmetric spaces.
Any Type~$\mathcal{A}$ model which is a symmetric space is linearly
isomorphic to $\mathcal{M}_1$, $\mathcal{M}_2$, or $\mathcal{M}_3$.
$\mathcal{M}_i$ is not linearly isomorphic to $\mathcal{M}_j$ for $i\ne j$.
$\mathcal{M}_1$ is isomorphic to $\mathcal{M}_2$ but not to $\mathcal{M}_3$.
\end{lemma}

A curve $\sigma$ is said to be a {\it geodesic} if it satisfies the {\it geodesic equation}
$\nabla_{\dot\sigma}\dot\sigma=0$. In any local coordinate chart, this means that
$$
\ddot x_\alpha^k+C_{ij}{}^k\dot x_\alpha^i\dot x_\alpha^j=0\,.
$$
An affine Type~$\mathcal{A}$ model $\mathcal{M}_{\mathcal{C}}$ is said to be {\it geodesically complete}
if any geodesic extends to have domain all of $\mathbb{R}$;
otherwise $\mathcal{M}_{\mathcal{C}}$ is said to be {\it geodesically incomplete}.
Conversely, $\mathcal{M}_{\mathcal{C}}$ is said to be {\it essentially geodesically complete} if there exists a Type~$\mathcal{A}$
surface which is modeled on $\mathcal{M}_{\mathcal{C}}$ which is geodesically complete;
otherwise $\mathcal{M}_{\mathcal{C}}$ is said to be {\it essentially geodesically incomplete}. It is clear that $\mathcal{M}_{\mathcal{C}}$
is geodesically complete implies $\mathcal{M}_{\mathcal{C}}$ is essentially geodesically complete and similarly
$\mathcal{M}_{\mathcal{C}}$ is essentially geodesically incomplete implies $\mathcal{M}_{\mathcal{C}}$ is geodesically incomplete.
It will follow from Theorem~\ref{T1.5} that up to linear equivalence, the only Type~$\mathcal{A}$ models which are geodesically
incomplete but essentially geodesically complete are $\mathcal{M}_1$ and $\mathcal{M}_3$.

\begin{ansatz}\rm\label{A1.4}
Let $\sigma_{a,b}(t):=(a,b)\cdot\log(t)$. The geodesic equations for a
Type~$\mathcal{A}$ model then reduce to the following pair of quadratic equations:
\begin{equation}\label{E1.a}\begin{array}{l}
a=C_{11}{}^1a^2+2C_{12}{}^1ab+C_{22}{}^1b^2,\\[0.05in]
b=C_{11}{}^2a^2+2C_{12}{}^2ab+C_{22}{}^2b^2\,.
\end{array}\end{equation}
Thus Equation~(\ref{E1.a}) has a solution with $(a,b)\ne(0,0)$ implies the associated Type~$\mathcal{A}$ model is geodesically incomplete;
this is a purely algebraic condition that can be studied using the quadratic formula.
\end{ansatz}

The following is the main result of this paper.

\begin{theorem}\label{T1.5}
Let $\mathcal{M}_{\mathcal{C}}=(\mathbb{R}^2,C_{ij}{}^k)$ be a Type~$\mathcal{A}$ model.
\begin{enumerate}
\item If the Ricci tensor has rank $1$, then:
\begin{enumerate}
\item If $\nabla R\ne0$, then $\mathcal{M}_{\mathcal{C}}$ is essentially geodesically incomplete.
\item $\mathcal{M}_2$ is geodesically complete.
\item $\mathcal{M}_1$ is geodesically incomplete and essentially geodesically complete.
\item $\mathcal{M}_3$ is geodesically incomplete. Let $\tilde{\mathcal{M}}_3$ be defined by taking $\Gamma_{22}{}^1=x^1$ and the remaining Christoffel symbols zero. $\tilde{\mathcal{M}}_3$ is modeled on $\mathcal{M}_3$ and is
geodesically complete. Thus $\mathcal{M}_3$ is essentially
geodesically complete.
\end{enumerate}
\item If the Ricci tensor has rank $2$, then following conditions are equivalent:
\begin{enumerate}
\item $\mathcal{M}_{\mathcal{C}}$ is geodesically incomplete.
\item $\mathcal{M}_{\mathcal{C}}$ is essentially geodesically incomplete.
\item There exists a geodesic of the form $\sigma=(a,b)\log(t)$ with $(a,b)\ne(0,0)$.
\item $\mathcal{M}_{\mathcal{C}}$ is not linearly equivalent to $\mathcal{M}_{-,\delta}$ for $0\le\delta<2$.
\end{enumerate}
\item No two models in the family
$\{\mathcal{M}_2,\mathcal{M}_3,\mathcal{M}_{-,\delta}\}_{0\le\delta<2}$ are locally isomorphic.
\end{enumerate}
\end{theorem}

The geodesic structures for $\mathcal{M}_2$,  $\mathcal{M}_{-,\delta}$ for $\delta=0$, $\delta=1.8$,
and $\tilde{\mathcal{M}}_3$ are depicted below. \medbreak\centerline{\includegraphics[height=5cm,keepaspectratio=true]{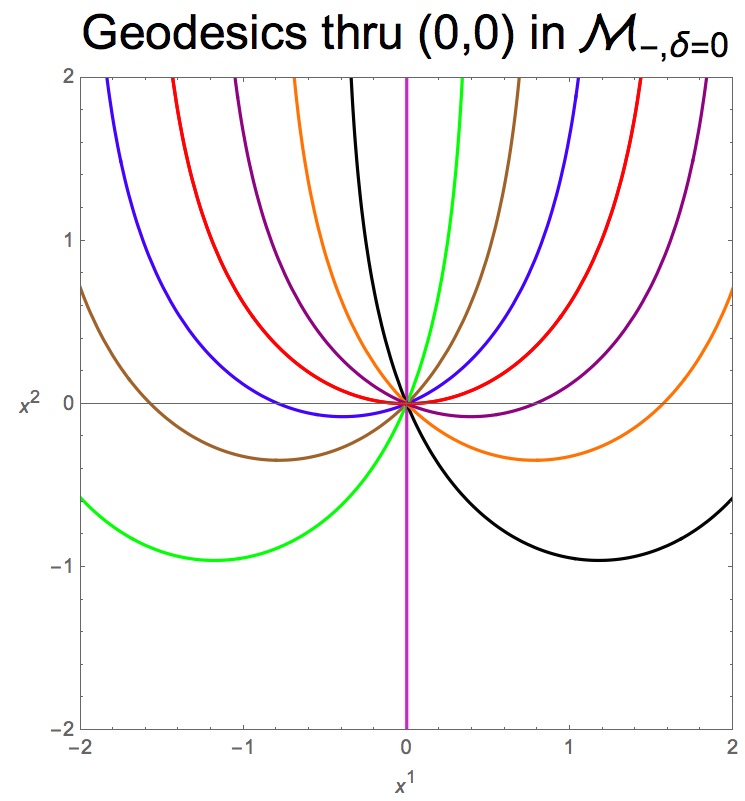}\quad
\includegraphics[height=5cm,keepaspectratio=true]{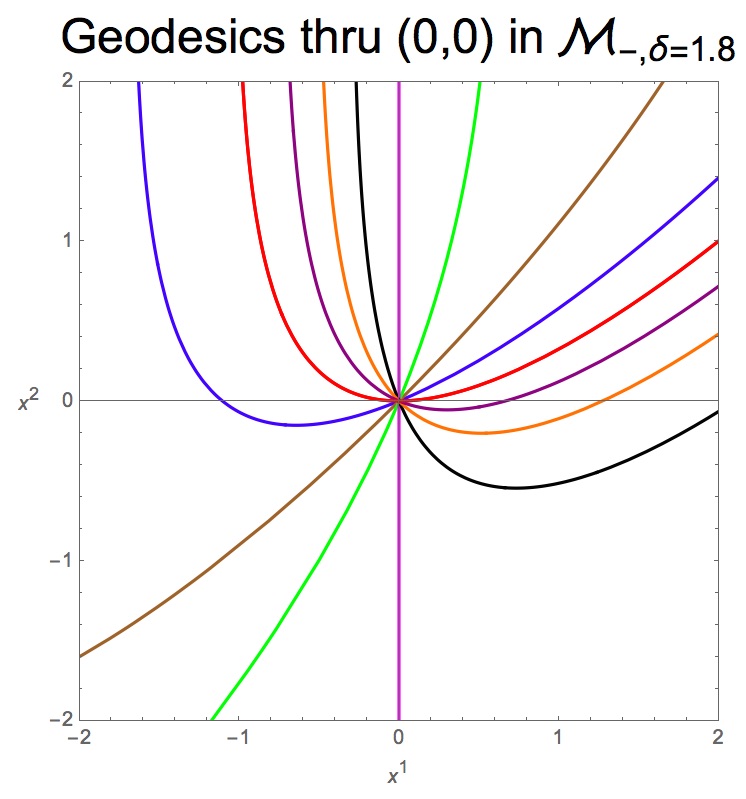}}
\medbreak\centerline{\includegraphics[height=5cm,keepaspectratio=true]{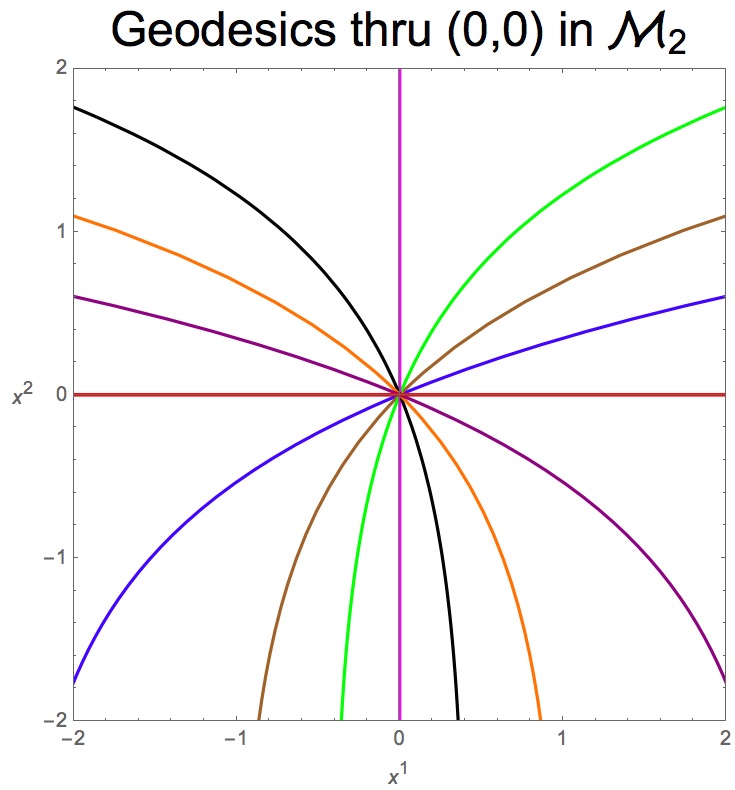}\quad
\includegraphics[height=5cm,keepaspectratio=true]{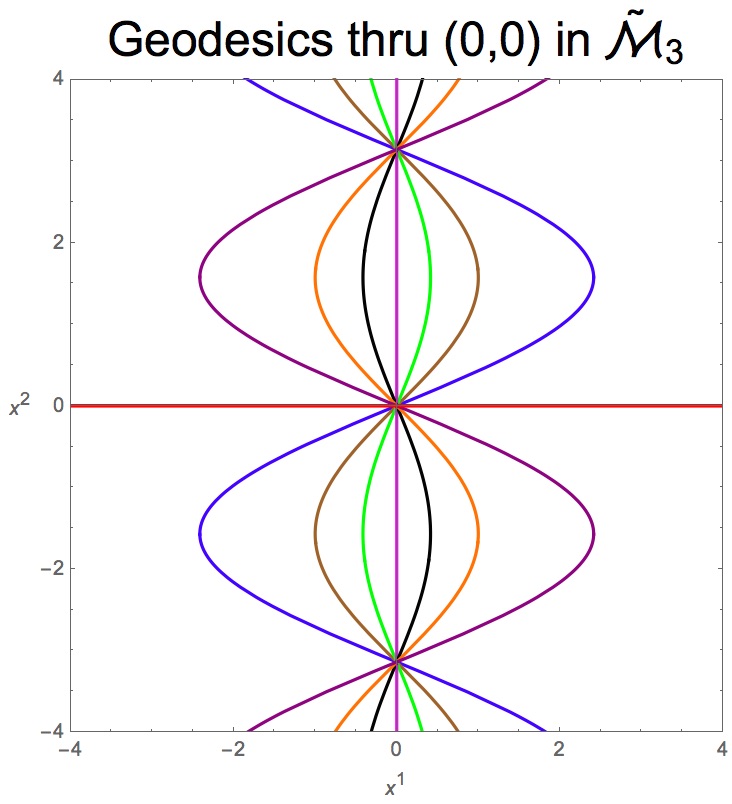}}

This gives rise to a simple algorithm to test if a given model $\mathcal{M}_{\mathcal{C}}$ is essentially
geodesically incomplete.
If the rank of the Ricci tensor is 1, then $\mathcal{M}_{\mathcal{C}}$ is essentially geodesically incomplete if and only if $\nabla\rho\ne0$.
If the rank of the Ricci tensor is 2, $\mathcal{M}$ is essentially geodesically incomplete if and only if there exists a non-trivial
geodesic of the form $\sigma(t)=(a,b)\cdot\log(t)$. Here is a brief outline to this paper.
In Section~\ref{S2} we prove Theorem~\ref{T1.5}~(1), in Section~\ref{S3}, we prove Theorem~\ref{T1.5}~(2),
and in Section~\ref{S4}, we prove Theorem~\ref{T1.5}~(3).

\section{Models where the rank of the Ricci tensor is 1}\label{S2}

We say that $X$ is an {\it affine Killing vector field} on an affine manifold $\mathcal{N}=(N,\nabla)$ if and
only if $\mathcal{L}_X(\nabla)=0$ where $\mathcal{L}$
denotes the Lie derivative; we refer to Kobayashi and Nomizu \cite[Chapter VI]{KN63} for additional characterizations of this condition. If $P\in N$, let
$\mathfrak{K}(P)$ be the space of germs of affine Killing vector fields based at $P$ and let $\mathfrak{K}$ be the space of global affine Killing vector
fields; $\mathfrak{K}(P)$ and $\mathfrak{K}$ are finite dimensional Lie algebras. If $N$ is simply connected and locally homogeneous, every germ of
an affine Killing vector field extends to a global Killing vector field so we may identify $\mathfrak{K}$ with $\mathfrak{K}(P)$ for any $P$.

\subsection{The proof of Theorem~\ref{T1.5}~(1a)}
Suppose that $\mathcal{M}_{\mathcal{C}}$ is a Type-$\mathcal{A}$ model where the Ricci tensor has rank 1 and
that $\mathcal{M}_{\mathcal{C}}$ is not a symmetric space. We argue as follows to show that
$\mathcal{M}_{\mathcal{C}}$ is essentially geodesically incomplete.

By Lemma~2.3 of \cite{BGGP16},  $\operatorname{Rank}(\rho)=1$ implies that we can make
a linear change of coordinates to ensure $C_{11}{}^2=0$, $C_{12}{}^2=0$, and $\rho=\rho_{22}dx^2\otimes dx^2$. Suppose that
$C_{22}{}^2=0$. A direct computation then shows that $\nabla\rho=0$. Consequently $C_{22}{}^2\ne0$. Let
$\sigma(t):=(x^1(t),(C_{22}{}^2)^{-1}\log(t))$. The geodesic equation
$\ddot x^2+C_{22}{}^2\dot x^2\dot x^2=0$ is then satisfied; the resulting ODE for $x^1$ can then be solved,
at least locally. Since $\rho$ has Rank $1$,  $\rho_{22}\ne0$. Let $\kappa(t):=\rho(\dot\sigma,\dot\sigma)=(tC_{22}{}^2)^{-2}\rho_{22}$;
the $x^1$ coordinate plays no role. We suppose there exists a geodesically complete Type~$\mathcal{A}$ affine surface $\mathcal{M}$
which is modeled on  $\mathcal{M}_{\mathcal{C}}$. We argue for a contradiction. Since $\mathcal{M}$ is modeled on $\mathcal{M}_{\mathcal{C}}$,
we can copy a portion of $\sigma(t)$ into $\mathcal{M}$. The function $\kappa(t)$ then extends to a real analytic function on all of $\mathbb{R}$
which is false.\qed

\medbreak Suppose that $\mathcal{M}_C$ is a symmetric space. By Lemma~\ref{L1.3}, $\mathcal{M}_{\mathcal{C}}$ is linearly
isomorphic to $\mathcal{M}_i$ for some $i$ with $1\le i\le 3$. We examine these 3 cases seriatim.

\subsection{The proof of Theorem~\ref{T1.5}~(1b) -- the model $\mathcal{M}_2$}
Since $C_{12}{}^1=-\frac12$ is the only non-zero Christoffel symbol,
the geodesic equations are $\ddot x^2=0$ and $\ddot x^1-\dot x^1\dot x^2=0$. Define a real analytic function on $\mathbb{R}^2$
by setting
$$
h(t;d):=\sum_{n=1}^\infty\frac{d^{n-1}t^n}{n!}=\left\{\begin{array}{ll}t&\text{ if }d=0,\\
\frac{e^{dt}-1}d&\text{ if }d\ne0\end{array}\right\}\,.
$$
We then have $h(0;d)=0$ and $\dot h(t;d)=e^{dt}$ so $\dot h(0;d)=1$. If $\sigma_{a,b,c,d}(t)$ is the solution to the geodesic equations with
$\sigma_{a,b,c,d}(0)=(a,b)$ and $\dot\sigma_{a,b,c,d}(0)=(c,d)$, then
$$
\sigma_{a,b,c,d}(t)=(a,b)+(c\cdot h(t;d),dt)\,.
$$
 Thus this geometry is geodesically complete. This proves Theorem~\ref{T1.5}~(1b). We have $\operatorname{Exp}_{a,b}(c,d)=(a,b)+(c\cdot h(1,d),d)$.
 Since $h(1,d)\ne0$ for any $d$, the exponential map is a diffeomorphism from the tangent space at $(a,b)$ to $\mathbb{R}^2$.  \qed

\subsection{The proof of Theorem~\ref{T1.5}~(1c) -- the model $\mathcal{M}_1$} Since the only non-zero Christoffel symbols are
$C_{11}{}^1=-1$ and $C_{12}{}^1=-\frac12$,
the geodesic equations become $\ddot x^1=\dot x^1\dot x^1+\dot x^1\dot x^2$ and $\ddot x^2=0$. This can be solved by taking
$\sigma(t)=(-1,0)\log t$. Thus this geometry is geodesically incomplete. Note that since $\rho(\dot\sigma,\dot\sigma)=0$,
the argument used to Theorem~\ref{T1.5}~(1a) does not apply. The fact that $\mathcal{M}_1$ and $\mathcal{M}_2$ are
locally isomorphic  follows from Lemma~\ref{L1.3}; thus $\mathcal{M}_1$ is essentially geodesically complete by Theorem~\ref{T1.5}~(1b). \qed

\subsection{The proof of Theorem~\ref{T1.5}~(1d) -- the model $\mathcal{M}_3$} Since the only non-zero Christoffel symbols are
$C_{11}{}^1=-1$ and $C_{22}{}^1=-1$,
the geodesic equations become $\ddot x^1=\dot x^1\dot x^1+\dot x^2\dot x^2$ and $\ddot x^2=0$. As for the geometry
$\mathcal{M}_1$, this can be solved by taking
$\sigma(t)=(-1,0)\log t$. Thus this geometry is geodesically incomplete. We complete the proof by showing $\tilde{\mathcal{M}}_3$ is
geodesically complete and modeled on $\mathcal{M}_3$.

We make a non-linear change of coordinates. Set $u^1=e^{-x^1}$ and $u^2=x^2$.
This defines a diffeomorphism $\Phi$ between $\mathbb{R}^2$
and $\mathbb{R}^+\times\mathbb{R}$. We have
\medbreak\qquad
$du^1=-e^{-x^1}dx^1,\quad du^2=dx^2,\quad\partial_1^u=-e^{x^1}\partial_1,\quad\partial_2^u=\partial_2$,
\medbreak\qquad $\nabla_{\partial_1^u}\partial_1^u= e^{x^1}\nabla_{\partial_1}\{e^{x^1}\partial_1\}
=e^{2x^1}\{(1+{}^x\Gamma_{11}{}^1)\partial_1+{}^x\Gamma_{11}{}^2\partial_2\}$,
\medbreak\qquad $\nabla_{\partial_1^u}\partial_2^u=-e^{x^1}\nabla_{\partial_1}\partial_2
=-e^{x^1}\{{}^x\Gamma_{12}{}^1\partial_1+{}^x\Gamma_{12}{}^2\partial_2\}$,
\medbreak\qquad $\nabla_{\partial_2^u}\partial_2^u=\nabla_{\partial_2}\partial_2
={}^x\Gamma_{22}{}^1\partial_1+{}^x\Gamma_{22}{}^2\partial_2$,
\medbreak\quad $\begin{array}{ll}{}^u\Gamma_{11}{}^1=-(1+{}^x\Gamma_{11}{}^1)\cdot e^{x^1}=0,&
{}^u\Gamma_{11}{}^2={}^x\Gamma_{11}{}^2\cdot e^{2x^1}=0,\\
{}^u\Gamma_{12}{}^1={}^x\Gamma_{12}{}^1=0,&
{}^u\Gamma_{12}{}^2=-{}^x\Gamma_{12}{}^2\cdot e^{x^1}=0,\\
{}^u\Gamma_{22}{}^1=-{}^x\Gamma_{22}{}^1\cdot e^{-x^1}= u^1,&
{}^u\Gamma_{22}{}^2={}^x\Gamma_{22}{}^2=0.
\end{array}$
\medbreak\noindent These are the Christoffel symbols defining $\tilde{\mathcal{M}}_3$. We have
$\tilde{\mathcal{M}}_3$ is isomorphic to $\mathcal{M}_3$ on $\mathbb{R}^+\times\mathbb{R}$ which is the range of the diffeomorphism.
We will complete the proof by showing that $\tilde{\mathcal{M}}_3$ is geodesically complete and that it is locally homogeneous.

The geodesic equations for $\tilde{\mathcal{M}}_3$ become $\ddot u^1+u^1\dot u^2\dot u^2 =0$ and $\ddot u^2=0$.
We solve these equations with initial conditions $\sigma_{a,b,c,d}(0)=(a,b)$ and $\dot\sigma_{a,b,c,d}(0)=(c,d)$ by taking
$$
\sigma_{a,b,c,d}(t):=\left\{\begin{array}{ll}
(a\cos(dt)+\frac cd\sin(dt),b+dt)&\text{ if }d\ne0\\
(a+ct,b)&\text{ if }d=0
\end{array}\right\}\,.
$$
This shows that $\tilde{\mathcal{M}}_3$ is complete. The exponential map at $(a,b)$ takes the form
$$\operatorname{Exp}_{a,b}(c,d)=\left\{\begin{array}{ll}(a\cos(d)+\frac cd\sin(d),b+d)&\text{ if }d\ne0\\
(a+c,b)&\text{ if }d=0\,.\end{array}\right\}
$$
If $d=\pi$, then $\operatorname{Exp}_{a,b}(c,\pi)=(-a,b+\pi)$ and the exponential map is neither 1-1 nor onto.

We use Lemma~3.6 of \cite{BGGP16} to see:
\begin{eqnarray*}
\mathfrak{K}(\mathcal{M}_{3})&=&\operatorname{Span}_{\mathbb{R}}
\{ \partial_1,\partial_2,e^{x^1}\cos(x^2)\partial_1,e^{x^1}\sin(x^2)\partial_1 \}\\
&=&\operatorname{Span}_{\mathbb{R}} \{ u^1\partial_1^u,\partial_2^u,\cos(u^2)\partial_1^u,\sin(u^2)\partial_1^u \}\,.
\end{eqnarray*}
Let
$$\begin{array}{lll}
\xi_1:=u^1 \partial_1^u,&\xi_2:= u^1\partial_1^u+\cos(u^2)\partial_1^u,&\xi_3:= u^1\partial_1^u-\sin(u^2)\partial_1^u,\\[0.05in]
\eta_1:=\partial_2^u,&\eta_2:=\partial_2^u+\sin(u^2)\partial_1^u,&\eta_3:=\partial_2^u+\cos(u^2)\partial_1^u.
\end{array}$$
These are affine Killing vector fields on $\mathbb{R}^+\times\mathbb{R}$. Since the structures are real analytic, these are affine Killing
vector fields on all of $\mathbb{R}^2$.
We verify that $[\xi_i,\eta_i]=0$. Fix $P$. If $ u^1(P)\ne0$, then $\xi_1(P)$ and $\eta_1(P)$ are linearly independent. By the
Frobenius Theorem, we can change to coordinates $w^1, w^2$ such that $\xi_1=\partial_1^w$ and $\eta_1=\partial_2^w$ near $P$. Since $\partial_1^w$ and $\partial_2^w$
are affine Killing vector fields, translations preserve the geometry and the geometry is Type~$\mathcal{A}$ near $P$ [Lemma 1 of \cite{AMK08}]. If $u^1(P)=0$ but $\cos( u^2(P))\ne0$, a similar argument pertains
using $\{\xi_2,\eta_2\}$. Finally, if $u^1(P)=0$ and $\cos(u^2(P))=0$, then necessarily $\sin(u^2(P))\ne0$ and
we can repeat the argument using $\{\xi_3,\eta_3\}$. This shows $\tilde{\mathcal{M}}_3$ is locally homogeneous. Since $\tilde{\mathcal{M}}_3$ is
modeled on $\mathcal{M}_3$ over $\mathbb{R}^+\times\mathbb{R}$, $\tilde{\mathcal{M}}_3$ is modeled on $\mathcal{M}_3$ everywhere.
If we use the diffeomorphism $\Phi$ to identify $\mathcal{M}_3$ with $\mathbb{R}^+\times\mathbb{R}$ in $\tilde{\mathcal{M}_3}$,
then the line $u^1(P)=0$ is the boundary of $\mathcal{M}_3$; this requires the delicate treatment given above.
\qed

\section{Models where the Ricci tensor has rank 2}\label{S3}

Theorem~\ref{T1.5}~(2) will follow from the following assertions:

\begin{lemma}\label{L3.1}
Let $\mathcal{M}_{\mathcal{C}}$ be a Type~$\mathcal{A}$ model with $\operatorname{Rank}(\rho)=2$ such that
 $\mathcal{M}_{\mathcal{C}}$ is geodesically incomplete. Then $\mathcal{M}_{\mathcal{C}}$ is essentially
geodesically incomplete.
\end{lemma}

\begin{lemma}\label{L3.2}
If $\mathcal{M}_{\mathcal{C}}$ is a Type~$\mathcal{A}$ model with $\operatorname{Rank}\{\rho\}=2$ which
does not admit a geodesic of the form $\sigma(t)=(a,b)\log(t)$ for $(a,b)\ne(0,0)$, then $\mathcal{M}_{\mathcal{C}}$
is linearly isomorphic to $\mathcal{M}_{-,\delta}$ for $0\le\delta<2$.
\end{lemma}

\begin{lemma}\label{L3.3}
If $0\le\delta<2$, then $\mathcal{M}_{-,\delta}$ is geodesically complete.
\end{lemma}

The remainder of this Section is devoted to the proof of these results; throughout it,
let $\mathcal{M}_{\mathcal{C}}$ be a Type~$\mathcal{A}$ affine model whose Ricci tensor has rank $2$.

\subsection{A genericity result}
We will use the following result to normalize our coordinate systems; this will simplify subsequent arguments
by avoiding the necessity to consider special cases.

\begin{lemma}\label{L3.4}
Suppose that $\operatorname{Rank}(\rho)=2$. Then we can change coordinates to ensure that $C_{ij}{}^k\ne0$ for all $i$, $j$, and $k$.\end{lemma}

\begin{proof} We want to show that there exists $T\in\operatorname{GL}(2,\mathbb{R})$ such that $(T^*C)_{ij}{}^k \neq 0$ for all $i$, $j$ and $k$. Suppose to the contrary that $(T^*C)_{22}{}^1=0$ for all $T\in\operatorname{GL}(2,\mathbb{R})$.
Set $T_\varepsilon(x^1,x^2):=(x^1+\varepsilon x^2,x^2)=(w^1,w^2)$. We compute
\begin{eqnarray*}
&&dw^1=dx^1+\varepsilon dx^2,\quad dw^2=dx^2,\quad\partial_{w^1}=\partial_{x^1},\quad\partial_{w^2}=-\varepsilon\partial_{x^1}+\partial_{x^2},\\
&&(T_\varepsilon^*C)_{22}{}^1=C(\partial_{w^2},\partial_{w^2},dw^1)\\
&&\quad=C_{22}{}^1+\varepsilon(C_{22}{}^2-2C_{12}{}^1)+\varepsilon^2(C_{11}{}^1-2C_{12}{}^2)
+ \varepsilon^3C_{11}{}^2\,.
\end{eqnarray*}
Since a non-trivial polynomial of degree at most $3$ has at most 3 roots, we obtain:
$$
C_{22}{}^1=0,\quad C_{22}{}^2=2C_{12}{}^1,\quad C_{11}{}^1=2C_{12}{}^2,\quad C_{11}{}^2=0\,.
$$
We compute
$$\rho=\left(\begin{array}{ll}
(C_{12}{}^2)^2&C_{12}{}^1C_{12}{}^2\\C_{12}{}^1C_{12}{}^2&(C_{12}{}^1)^2
\end{array}\right)\,.
$$
Thus $\det(\rho)=0$ so the Ricci tensor does not have rank 2 which is false.
Thus we obtain $(T_\varepsilon^*C)_{22}{}^1=0$ for at most 3 values of $\varepsilon$. Change coordinates
to ensure $C_{22}{}^1\ne0$. Let $S_\varepsilon(x^1,x^2):=(x^1,\varepsilon x^1+x^2)=(u^1,u^2)$. We compute
\begin{eqnarray*}
&&du^1=dx^1,\quad du^2=\varepsilon dx^1+dx^2,\quad\partial_{u^1}=\partial_{x^1}-\varepsilon\partial_{x^2},\quad\partial_{u^2}=\partial_{x^2},\\
&&(S_\varepsilon^*C)_{22}{}^1=C_{22}{}^1,\quad(S_\varepsilon^*C)_{22}{}^2=C_{22}{}^2+\varepsilon C_{22}{}^1,\\
&&(S_\varepsilon^*C)_{12}{}^1=C_{12}{}^1-\varepsilon C_{22}{}^1,\\
&&(S_\varepsilon^*C)_{12}{}^2=C_{12}{}^2+\varepsilon(C_{12}{}^1- C_{22}{}^2)-\varepsilon^2C_{22}{}^1,\\
&&(S_\varepsilon^*C)_{11}{}^1=C_{11}{}^1-2\varepsilon C_{12}{}^1+\varepsilon^2C_{22}{}^1,\\
&&(S_\varepsilon^*C)_{11}{}^2=C_{11}{}^2+\varepsilon(C_{11}{}^1-2C_{12}{}^2)+\varepsilon^2(C_{22}{}^2-2C_{12}{}^1)+\varepsilon^3C_{22}{}^1\,.
\end{eqnarray*}
Since $C_{22}{}^1\ne0$, all these polynomials are non-trivial. There are only a finite number of zeros of each of these polynomials. Thus
for generic $\varepsilon$, all the Christoffel symbols are non-zero.
\end{proof}

\subsection{Geodesics of the form $\sigma(t)=(a,b)\log(t)$ for $(a,b)\ne(0,0)$}
The geodesic equations for the curve $\sigma$ are given in Equation~(\ref{E1.a}).
If we can find a solution to these equations with $(a,b)\ne(0,0)$, then $\mathcal{M}$ is geodesically incomplete. Set
$$
E_i(a,b):=a^2C_{11}{}^i+2abC_{12}{}^i+b^2C_{22}{}^i
$$
to rewrite Equation~(\ref{E1.a}) in the form:
\begin{equation}\label{E3.a}
a=E_1(a,b)\text{ and }b=E_2(a,b)\,.
\end{equation}
We assume $a\ne0$ and set $b=\lambda a$. Equation~(\ref{E3.a}) becomes
$$
a=a^2E_1(1,\lambda)\text{ and }a\lambda=a^2E_2(1,\lambda) \text{ i.e. }1=aE_1(1,\lambda)\text{ and }\lambda=aE_2(1,\lambda)\,.
$$
We eliminate $a$ in these equations to obtain
$a=\frac1{E_1(1,\lambda)}$ and $a=\frac\lambda{E_2(1,\lambda)}$. For this ansatz to work, we need to be able to solve the equations:
\begin{equation}\label{E3.b}
E_3(\lambda):=\lambda E_1(1,\lambda)-E_2(1,\lambda)=0,\quad E_1(1,\lambda)\ne0,\quad E_2(1,\lambda)\ne0\,.
\end{equation}

\begin{lemma}\label{L3.5} Normalize the coordinate system as in Lemma~\ref{L3.4}.
Then $E_1(1,\lambda)$ and $E_2(1,\lambda)$ have degree 2 in $\lambda$.
If $E_1(1,\lambda)$ and $E_2(1,\lambda)$ do not have a common real root, then the geometry of the underlying Type~$\mathcal{A}$
model is geodesically incomplete.
\end{lemma}

\begin{proof} We have
\begin{equation}\label{E3.c}\begin{array}{l}
E_1(1,\lambda)=C_{11}{}^1+2\lambda C_{12}{}^1+\lambda^2 C_{22}{}^1,\\[0.05in]
E_2(1,\lambda)=C_{11}{}^2+2\lambda C_{12}{}^2+\lambda^2C_{22}{}^2,\\[0.05in]
E_3(\lambda)=-C_{11}{}^2+\lambda(C_{11}{}^1- 2 C_{12}{}^2)+\lambda^2(2C_{12}{}^1-C_{22}{}^2)+\lambda^3C_{22}{}^1\,.
\end{array}\end{equation}
Since $C_{22}{}^1\ne0$ and $C_{22}{}^2\ne0$, $E_1(1,\lambda)$ and $E_2(1,\lambda)$ have degree 2 in $\lambda$
while $E_3(\lambda)$ has degree $3$ in $\lambda$. Furthermore, since $C_{11}{}^1\ne0$
and $C_{11}{}^2\ne0$, $\lambda=0$ is not a root of $E_1(1,\lambda)$, $E_2(1,\lambda)$, or $E_3(\lambda)$.
Choose a real root $\lambda_0\ne0$ of the cubic $E_3(\lambda)$.
If $E_1(1,\lambda_0)=0$, then $E_2(1,\lambda_0)=0$ so $E_1(1,\lambda)$ and $E_2(1,\lambda)$ have a common root which is false.
If $E_2(1,\lambda_0)=0$, since $\lambda_0\ne0$, $E_1(1,\lambda_0)=0$ which again is false. Thus we can satisfy the conditions
of Equation~(\ref{E3.b}) which implies the geometry is geodesically incomplete.
\end{proof}

\subsection{The proof of Lemma~\ref{L3.1}}
Let $\mathcal{M}_{\mathcal{C}}$ be a Type~$\mathcal{A}$ model with $\operatorname{Rank}(\rho)=2$ such that
 $\mathcal{M}_{\mathcal{C}}$ is geodesically incomplete. We wish to show that $\mathcal{M}_{\mathcal{C}}$ is essentially
geodesically incomplete. Suppose to the contrary that there exists a Type~$\mathcal{A}$ surface $\mathcal{M}=(M,\nabla)$
which is geodesically complete and which is modeled on $\mathcal{M}_{\mathcal{C}}$ and argue for a contradiction. By passing
to the universal cover, we may assume $M$ is simply connected and therefore that any local affine Killing vector field extends
to a global affine Killing vector field.

Since the Ricci tensor has rank 2, Theorem~3.4 of Brozos-V\'azquez et.\ al \cite{BGGP16} shows that
$\mathfrak{K}(\mathcal{M}_{\mathcal{C}})=\operatorname{Span}_{\mathbb{R}}\{\partial_{x^1},\partial_{x^2}\}$.
Let $P\in M$, and let $\{\xi_1,\xi_2\}$ be a basis for $\mathfrak{K}(P)$. Extend the $\xi_i$ to globally defined affine Killing vector fields. Since $\dim\{\mathfrak{K}(Q)\}=2$ for any $Q\in M$, we conclude that
$\{\xi_1(Q),\xi_2(Q)\}$ is a basis for $\mathfrak{K}(Q)$; $\{\xi_1,\xi_2\}$ is a global frame for $TM$.
Since the model $\mathcal{M}_{\mathcal{C}}$ is incomplete, there exists a geodesic $\sigma(t)$ in $\mathcal{M}_{\mathcal{C}}$
which is not defined for all $t$; we suppose $(a,b)$ for $b<\infty$ to be a maximal parameter range. Copy a piece of $\sigma$ into $\mathcal{M}$
to define a geodesic $\sigma_{\mathcal{M}}$. We may assume, without loss of generality, that $\xi_i(\sigma_{\mathcal{M}}(0))=\partial_i$
and hence $\xi_i(\sigma_{\mathcal{M}})(t)=\partial_i$ for $t$ near $0$.
Since $\mathcal{M}$ is geodesically complete, the domain of $\sigma_{\mathcal{M}}$ is $\mathbb{R}$.
Expand $\dot\sigma_{\mathcal{M}}=\kappa_{1,\mathcal{M}}(t)\xi_1(\sigma(t))+\kappa_{2,\mathcal{M}}(t)\xi_2(\sigma(t))$.
The functions $\kappa_{i,\mathcal{M}}(t)$ are then
real analytic and defined on all of $\mathbb{R}$. Returning to the model $\mathcal{M}_{\mathcal{C}}$, we may expand
$\dot\sigma=\kappa_1(t)\partial_1+\kappa_2(t)\partial_2$; since we are working in the real analytic context and since
$\xi_i(\sigma_{\mathcal{M}})(0)=\partial_i$, $\kappa_i(t)=\kappa_{i,\mathcal{M}}(t)$ so $\kappa_i$ extends to a real analytic function on all of $\mathbb{R}$. In particular,
the $\kappa_i$ extend smoothly to $t=b$. Since
$$
\lim_{t\rightarrow b}\sigma(t)=\sigma(0)+\int_{t=0}^b(\kappa_1(t),\kappa_2(t))dt\,,
$$
we conclude that $\sigma$ is smooth at $t=b$ and hence $\sigma$ extends beyond $t=b$ which is false. This contradiction shows
that in fact $\mathcal{M}_{\mathcal{C}}$ is essentially geodesically incomplete.
\qed

\subsection{The proof of Lemma~\ref{L3.2}}

Let $\mathcal{M}_{\mathcal{C}}$ be a Type~$\mathcal{A}$ model with $\operatorname{Rank}\{\rho\}=2$ which
does not admit a geodesic of the form $\sigma(t)=(a,b)\log(t)$ for $(a,b)\ne(0,0)$. We must show that $\mathcal{M}_{\mathcal{C}}$
is linearly isomorphic to $\mathcal{M}_{-,\delta}$ for $0\le\delta<2$. Normalize the coordinate system as in Lemma~\ref{L3.4}.
The ansatz of Lemma~\ref{L3.5} fails and therefore $E_1(1,\lambda)$ and $E_2(1,\lambda)$ have a common root.
Suppose $E_1(1,\lambda_0)=0$ and $E_2(1,\lambda_0)=0$. We may factor
$$
E_1(1,\lambda_0)=C_{22}{}^1(\lambda-\lambda_0)(\lambda-\lambda_1)\text{ and }
E_2(1,\lambda_0)=C_{22}{}^2(\lambda-\lambda_0)(\lambda-\lambda_2)\,.
$$
We use Equation~(\ref{E3.c}) to determine the Christoffel symbols and express the geodesic equation in the form:
\begin{equation}\label{E3.d}
\ddot x^1+C_{22}{}^1(\dot x^2-\lambda_0\dot x^1)(\dot x^2-\lambda_1\dot x^1) =0\text{ and }
\ddot x^2+C_{22}{}^2(\dot x^2-\lambda_0\dot x^1)(\dot x^2-\lambda_2\dot x^1)=0\,.
\end{equation}
We set $u:=-\dot x^1$ and $v:=-\dot x^2$ to work in phase space. Equation~(\ref{E3.d}) can be rewritten in the form:
$$
\dot u=E_1(u,v)=\xi_1(u,v)\xi_2(u,v)\text{ and }\dot v=E_2(u,v)=\xi_1(u,v)\xi_3(u,v)$$
where $\xi_i(u,v)=\alpha_iu+\beta_iv$ and $(\alpha_i,\beta_i)\ne(0,0)$.
We may change variables so that $u_1=\xi_1(u,v)$ and $v_1$ is chosen suitably (we could, for example,
take $v_1=-\beta_1u+\alpha_1v$ but the choice is irrelevant for the moment). The geodesic equations become
$$
\dot u_1=u_1\eta_2(u_1,v_1)\text{ and }\dot v_1=u_1\eta_3(u_1,v_1)
\text{ where }\eta_i(u,v)=\alpha_{1,i}u+\beta_{1,i}v\,.
$$
Suppose that $\beta_{1,2}=0$ so in the new coordinate system we have that
$C_{12}{}^1=0$, and $C_{22}{}^1=0$. This implies
$\rho=(C_{11}{}^1 C_{12}{}^2+C_{11}{}^2 C_{22}{}^2-C_{12}{}^2C_{12}{}^2)dx^1\otimes dx^1$
which is false as $\rho$ has rank $2$. Consequently $\beta_{1,1}\ne0$ so
$\{u_1,\eta_2(u_1,v_1)\}$ are linearly independent linear functions. We change variables
setting $u_2=u_1$ and $v_2=\eta_2(u_1,v_1)$. This permits us to write the geodesic equations in the form:
$$
\dot u_2=u_2v_2\text{ and }\dot v_2=u_2(\tilde\alpha u_2+\tilde\beta v_2)\,.
$$
If $\tilde\alpha=0$ and $\tilde\beta=0$, then in the new coordinate system $C_{11}{}^2=0$ and $C_{12}{}^2=0$. This implies
$ \rho=(-C_{12}{}^1C_{12}{}^1+ C_{11}{}^1 C_{22}{}^1 + C_{12}{}^1 C_{22}{}^2)dx^2\otimes dx^2$, which contradicts our assumption that $\rho$ has rank 2. So if $\tilde\alpha=0$, then $\tilde\beta\neq 0$ and, by rescaling $u_2$ appropriately and dropping the subscripts to simplify the notation, the geodesic equations can be written in the form $\ddot x^1=\ddot x^2 = \dot x^1\dot x^2$, which admit the solution $\sigma(t)=(-1,-1)\log(t)$, in contradiction with our assumption. Thus $\tilde\alpha\neq 0$, the
geodesic equations have the form
$$\ddot x^1+\dot x^1\dot x^2=0\text{ and }\ddot x^2+\dot x^1(C_{11}{}^2\dot x^1+C_{12}{}^2\dot x^2)=0\,,$$
and we can rescale $x^1$ to ensure $C_{11}{}^2=\pm1$, showing that
$\mathcal{M}_{\mathcal{C}}$ is linearly isomorphic to $\mathcal{M}_{\pm,\delta}$.

We examine $\mathcal{M}_{+,\delta}$. Let $\sigma(t)=(a,1)\cdot\log(t)$. The geodesic equations become
$-a+a=0$ and $-1+a^2+\delta a=0$. We use the quadratic formula to solve the second equation setting
$a=\frac12(-\delta\pm\sqrt{\delta^2+4})$. Thus this possibility is eliminated.

We examine $\mathcal{M}_{-,\delta}$ with $\delta^2\ge4$. Let $\sigma(t)=(a,1)\cdot\log(t)$. This time the geodesic
equations become
$-a+a=0$ and $-1-a^2+\delta a=0$. The quadratic formula yields $a=\frac12( \delta\pm\sqrt{\delta^2-4})$. Thus this possibility
is eliminated if $\delta^2\ge4$.

Thus $\mathcal{M}_{\mathcal{C}}$ is linearly equivalent to $\mathcal{M}_{-,\delta}$ for $\delta^2<4$.
By replacing $x^1$ by $-x^1$ if necessary, we can always assume $\delta\ge0$ and thus $0\le\delta<2$.
\qed

\subsection{The proof of Lemma~\ref{L3.3}}
We must show that if $0\le\delta<2$, then $\mathcal{M}_{-,\delta}$ is geodesically complete.
Set $u=-\dot x^1$ and $v=-\dot x^2$ so we work
in phase space. Let $\mathfrak{X}(u,v):=(uv,u(-u+\delta v))$. Then $\sigma$ is a geodesic if and only if $(\dot u,\dot v)=\mathfrak{X}(u,v)$.
Thus we are examining the flow curves of the vector field $\mathfrak{X}$.
As a guide to the intuition, we present a picture of the flow curves of the vector field $\mathfrak{X}$ when $\delta=1$.
\smallbreak\centerline{\includegraphics[width=6cm,keepaspectratio=true]{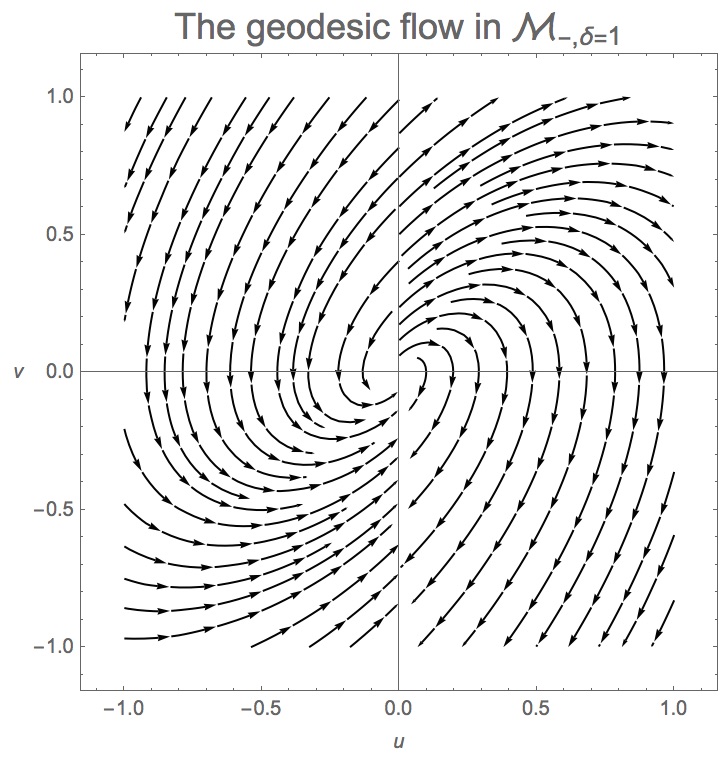}}
\medbreak\noindent Let $P\in\mathbb{R}^2$. By the Fundamental Theorem of Ordinary Differential Equations, there exists
a unique curve $\gamma=\gamma_P$ defined for $|t|<\varepsilon$ so that $\dot\gamma_P=\mathfrak{X}(\gamma_P)$ and
$\gamma_P(0)=P$. The corresponding geodesic $\sigma$ is then found by solving the ODE $\dot\sigma=\gamma$ with
an appropriate initial condition. Note that if $\mathfrak{X}(P)=0$, then we may take $\gamma_P(t)=P$ to be the constant curve. Consequently
if $\gamma(t)$ is a flow curve for $\mathfrak{X}$ and if $\mathfrak{X}(\gamma(t))=0$ for any $t$, then $\gamma(t)$ is the constant curve.
Note that $\mathfrak{X}(u,v)=0$ if and only if $u=0$. Thus flow curves can not cross the vertical axis; either $u(\gamma(t))>0$ or
$u(\gamma(t))<0$ for all $t$ in the domain or $(u(t),v(t))$ is constant.

If $v=0$, then $\dot v=-u^2$. Thus
$v$ is strictly monotonically decreasing near the horizontal axis. Once a flow curve has left the first quadrant, it is trapped in the fourth quadrant.
Similarly once a flow curve has left the second quadrant it is trapped in the third quadrant.
The picture given above suggests that positive vertical axis is a repulsive fixed point set and the negative vertical axis is an attractive fixed point set;
flow curves should exist for all time and pass from the positive vertical axis to the negative vertical axis. This is in fact the case as we now show.

Let $\alpha:=\frac{\dot v}{\dot u}=\frac{-u+\delta v}{v}$ be the slope of the flow curve. A direct computation shows that
$$\begin{array}{l}
\dot\alpha=v^{-2}\left\{(-\dot u+\delta\dot v)v-(-u+\delta v)\dot v\right\}\\[0.05in]
\quad=v^{-2}\left\{-uv^2+u^2(-u+\delta v)\right\}\\[0.05in]
\quad=v^{-2}u\left\{-v^2-u^2+ \delta uv\right\}\,.
\end{array}$$
The assumption that $0\le\delta<2$ permits us to estimate using the  Cauchy-Schwarz inequality that
$u(t)^2+v(t)^2- \delta u(t)v(t)\ge\epsilon v(t)^2$ for some $\epsilon>0$. Consequently:
\begin{equation}\label{E3.e}
\dot\alpha(t)\le -\varepsilon|u(t)|\text{ if }u(t)>0\text{ and }v(t)\ne0\,.
\end{equation}
We examine the behavior of the flow curves in each quadrant; the first quadrant is the most difficult to study. We suppose that $\gamma$ is a flow curve
for the vector field $(uv,u(-u+\delta v))$ with initial condition $\gamma(0)=(u_0,v_0)$. If $u_0=0$, then $\gamma$ is constant. Thus we assume
$u_0\ne0$. We assume $\gamma$ is incomplete and let $[0,T)$ be a maximal domain. If the range of $\gamma$ is trapped in a compact set $K$, then
there exist $t_n\rightarrow T$ so $\gamma(t_n)\rightarrow(u_T,v_T)$ exists. Since any flow curve with initial condition near $(u_T,v_T)$ extends for a
fixed minimal time, this would permit us to extend $\gamma$ past $t=T$ which is impossible. Thus $\gamma$ must escape to infinity.

\subsection*{Case 3.1. The first quadrant} We suppose that $u(0)=u_0>0$, $v(0)=v_0>0$, and $\alpha(0)=\alpha_0>0$.
We suppose these conditions pertain on all of $[0,T)$ and argue for a contradiction. In the first quadrant we have $\dot u>0$ and consequently $u$
is monotone increasing. We apply Equation~(\ref{E3.e}). The slope of the flow curve is monotonically decreasing.
The slope is positive by assumption when $t=0$.
Thus the flow line lies under the tangent line and we have an estimate of the form
\begin{equation}\label{E3.f}
v(t)< v_0+ \alpha_0u(t)\,.
\end{equation}
We have $u$ is monotone increasing. If $u$ remains bounded as $t\rightarrow T$, then Equation~(\ref{E3.f}) shows that $v$ is bounded from
above. Since by hypothesis $v$ is non-negative, both $u$ and $v$ are bounded so the curve is trapped in a compact region which is false.
Thus $\lim_{t\rightarrow T}u(t)=\infty$. We use Equation~(\ref{E3.f}) to estimate
$$
\textstyle \frac1{v(t)}>\frac1{v_0+\alpha_0u(t)}\text{ and }-\frac1{v(t)} <-\frac1{v_0+\alpha_0u(t)}\,.
$$
Since $u$ is monotone increasing, we can use the parameter $s=u(t)$. We have
\begin{eqnarray*}
\partial_s\alpha&=&\frac{\dot\alpha(t)}{\dot u(t)}\le-\frac\varepsilon{v(t)}\le-\frac\varepsilon{v_0+\alpha_0s},\\
\lim_{t\rightarrow T}\alpha(t)&=&\lim_{s\rightarrow\infty}\alpha(t(s))=\alpha(0)+\int_{s=u_0}^\infty\partial_s\alpha ds\\
& \leq&\alpha(0)-\int_{s=u_0}^\infty\frac\varepsilon{v_0+\alpha_0s}ds=-\infty\,.
\end{eqnarray*}
Thus it is not possible that $\alpha(t)>0$ for all $t\in[0,T)$ and we must have $\alpha(t)=0$ for some $t\in[0,T)$.

We restart the curve with $u(0)=u_0>0$, $v(0)=v_0>0$, and $\alpha(0)=\alpha_0\le0$. Since the slope is monotone decreasing,
we may restart the process and assume in fact $\alpha(0)<0$. Suppose that $v(t)>0$ on $[0,T)$. Since $\alpha$ is monotone decreasing,
this implies the curve is trapped in the triangle bounded by the positive vertical axis, the positive horizontal
axis, and the tangent line which has negative slope. This is impossible. Consequently, the curve crosses the positive horizontal axis
and escapes into the fourth quadrant where it is trapped.

\subsection*{Case 3.2. The fourth quadrant.} We have $\dot u=uv<0$ and $\dot v=u(-u+\delta v)<0$. Thus both $u$ and
$v$ are monotone decreasing. The slope of the tangent line is positive and decreasing. Thus the curve is trapped above the tangent line,
below the positive horizontal axis, and to the right of the negative vertical axis. This is a compact region so this is impossible.

\subsection*{Case 3.3. The second and third quadrants.} Suppose $u(t)<0$. We compute:
\begin{eqnarray*}
&&\partial_t\{u^2(t)+v^2(t)\}=2u(t)\dot u(t)+2v(t)\dot v(t)\\
&=&2u(t)u(t)v(t)+2v(t)(-u(t)u(t)+\delta u(t)v(t))=2\delta u(t)v(t)^2\le0\,.
\end{eqnarray*}
Thus the radial distance to the origin is non-increasing and the curve is trapped in a quarter circle which is impossible.
\qed

\begin{remark}\rm
The analysis of Bromberg et al.\ \cite{B05} gives a criterion for examining when a quadratic vector field is complete; this is
clearly relevant to the study we presented in Lemma~\ref{L3.3} and parallels the algorithm we used there. We chose
to present an independent derivation as the focus of this paper is quite different. We wished to show that the ansatz of
considering geodesics of the form $\sigma(t)=(a,b)\cdot\log(t)$ gave a complete answer to the question of geodesic completeness;
a Type-$\mathcal{A}$ model is geodesically incomplete if and only if there exists a geodesic of this form, i.e.\ Equation~(\ref{E1.a}) can
be satisfied for $(a,b)\ne0$. We also wished to study the relationship between geodesic incompleteness, essential geodesic incompleteness,
and symmetric geometry. Thus we needed a more refined geometric analysis than is presented by Bromberg et al.\ and, in any event,
we wished to keep this paper as self-contained as possible.
\end{remark}

\section{The moduli space}\label{S4}
The geometries $\mathcal{M}_{-,\delta}$ for $0\le\delta<2$ form a 1-parameter family of geodesically complete
models such that the Ricci tensor has rank 2 and signature $(1,1)$; thus these are distinct from the models $\mathcal{M}_2$ and $\mathcal{M}_3$
where the Ricci tensor has rank 1.
The moduli spaces of Type~$\mathcal{A}$ structures with non-singular Ricci tensor were examined in \cite{BGGP16a} where a complete
set of invariants was given. Let
$$
\check\rho_{ij}:=\Gamma_{ik}{}^l\Gamma_{jl}{}^k,\quad
\Sigma:=\operatorname{Tr}_\rho\{\check\rho\}=\rho^{ij}\check\rho_{ij},\quad
\Psi:=\det(\check\rho)/\det(\rho).
$$
Assume the Ricci tensor has rank 2.
By Lemma~\ref{L1.2} the structure group is $\operatorname{GL}(2,\mathbb{R})$. Consequently, $\Sigma$ and $\Psi$ are affine invariants.
It was shown in \cite{BGGP16a} that $(\Sigma,\Psi)$ together
with the signature of the Ricci tensor form a complete set of invariants for the associated moduli spaces. Thus, for example,
 two Type~$\mathcal{A}$ moduli spaces $\mathcal{M}_i$ with indefinite Ricci tensor are isomorphic
if and only if
$$(\Sigma(\mathcal{M}_1),\Psi(\mathcal{M}_1))=(\Sigma(\mathcal{M}_2),\Psi(\mathcal{M}_2))\,.
$$
We compute that $(\Sigma(\mathcal{M}_{-,\delta}),\Psi({\mathcal{M}_{-,\delta}}))=(-3+2\delta^2,2)$.
Thus, in particular, $\mathcal{M}_{-,\delta}$ is isomorphic to $\mathcal{M}_{-,\tilde\delta}$ if and only if $\delta=\pm\tilde\delta$.
Since the Ricci tensor of $\mathcal{M}_2$ is negative semi-definite, the Ricci tensor of $\mathcal{M}_3$ is positive semi-definite,
and the Ricci tensor of $\mathcal{M}_{-,\delta}$ has signature $(1,1)$, we may conclude that no two models in the family
$\{\mathcal{M}_2,\mathcal{M}_3,\mathcal{M}_{-,\delta}\}$ for $0\le\delta<2$ are locally isomorphic.
We show below the moduli space $\mathfrak{M}(1,1)$ of Type~$\mathcal{A}$ affine surfaces where the Ricci tensor is indefinite, i.e. has signature $(1,1)$.
The moduli space is the simply connected region of the plane which
is bounded on the left (resp. right) by the curve $\sigma_-(t)$ (resp. $\sigma_+(t)$) where
$$
\sigma_\pm(t):=\left(\pm4 t^2\pm\frac{1}{t^2}+2,4 t^4\pm4 t^2+2\right)\,.
$$
We indicate below the segment $(-3+\delta^2,2)$ for $0\le\delta<2$ as a thick black segment;
the far left hand endpoint corresponds to $\delta=0$. We have also shown below the segment $(-3+\delta^2,2)$ for $2\le\delta$ as a dashed curve;
it is an asymptotic lower bound to the bounding curve $\sigma_+(t)$; $\sigma_+(t)\rightarrow 2$ as $t\rightarrow0$.
\medbreak\centerline{\includegraphics[width=5cm,keepaspectratio=true]{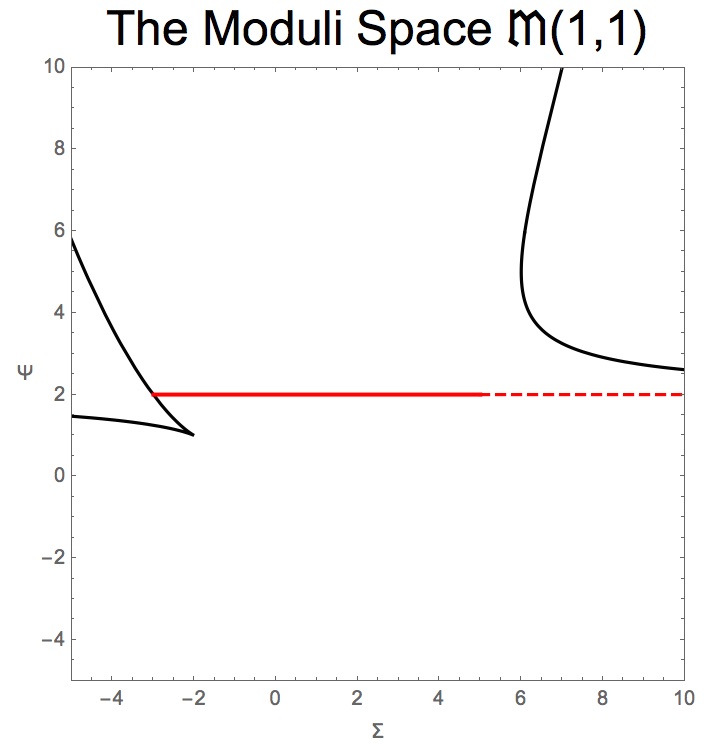}}

\subsection*{Acknowledgments}
Research of the authors was partially supported by project GRC2013-045 (Spain), by project PIP 1787/681 (CONICET, Argentina) and by project 11/X615 (UNLP, Argentina).

\end{document}